\documentclass{article}

\usepackage{amsmath,amsfonts,amsthm,amssymb,graphicx,tikz,tikz-cd}
\usepackage[all,2cell,ps]{xy}

\theoremstyle{plain}
\newtheorem{thm}{Theorem}[section]

\theoremstyle{definition}

\newtheorem{rem}{Remark}

\newcommand{\Z}{\mathbb Z}

\newcommand{\Q}{\mathbb Q}
\newcommand{\B}{\mathbb B}
\newcommand{\C}{\mathbb C}
\newcommand{\pp}{\mathbb{P}}

\newcommand{\al}{\alpha}

\newcommand{\Lam}{\Lambda}
\newcommand{\Gam}{\Gamma}

\newcommand{\Del}{\Delta}
\newcommand{\lam}{\lambda}

\DeclareMathOperator{\PU}{PU}

\newcommand{\ssm}{\smallsetminus}

\newenvironment{pf}{\begin{proof}}{\end{proof}}

\usetikzlibrary{arrows}

\title{Bielliptic ball quotient compactifications and lattices in $\text{PU}(2, 1)$ with finitely generated commutator subgroup}
\author{Luca F.\ Di Cerbo\footnote{This material is based upon work supported by a Grant of the Max Planck Society: ``Complex Hyperbolic Geometry and Toroidal Compactifications''.} \\ \small{Max Planck Institute for Mathematics} \\ \small{\textsf{luca@mpim-bonn.mpg.de}} \and Matthew Stover\footnote{This material is based upon work supported by the National Science Foundation 
under Grant Number NSF DMS-1361000. The second author acknowledges support from U.S.\ National Science Foundation grants DMS 1107452, 1107263, 1107367 ``RNMS: GEometric structures And Representation varieties'' (the GEAR Network).} \\ \small{Temple University}\\ \small{\textsf{mstover@temple.edu}}}
\date{\today}

\begin{document}

\maketitle

\begin{abstract}
We construct two infinite families of ball quotient compactifications birational to bielliptic surfaces. For each family, the volume spectrum of the associated noncompact finite volume ball quotient surfaces is the set of all positive integral multiples of $\frac{8}{3}\pi^{2}$, i.e., they attain all possible volumes of complex hyperbolic $2$-manifolds. The surfaces in one of the two families all have $2$-cusps, so that we can saturate the entire volume spectrum with $2$-cusped manifolds. Finally, we show that the associated neat lattices have infinite abelianization and finitely generated commutator subgroup. These appear to be the first known nonuniform lattices in $\PU(2,1)$, and the first infinite tower, with this property.
\end{abstract}

\section{Introduction}\label{sec:Intro}

A complete complex hyperbolic surface, or a ball quotient surface, is a complete Hermitian surface of constant holomorphic sectional curvature $-1$. More precisely, if $\B^2$ denotes the unit ball in $\C^2$ equipped with the normalized Bergman metric, any ball quotient surface is of the form $Y = \B^2 / \Gam$ where $\Gam \subset \PU(2,1)$ a torsion-free lattice. When $Y$ is noncompact and of finite volume, then $Y$ is a complex $2$-manifold with a finite number of topological ends called the cusps of $Y$. The Baily--Borel or minimal compactification $Y^*$ of $Y$ is a normal projective surface with finitely many singular points in one-to-one correspondence with the cusps of $Y$. If $\Gam$ is neat (see \S \ref{ssec:Ball}), then $Y^{*}$ admits a particularly nice minimal resolution of singularities $X$ that is a smooth projective surface called the smooth toroidal compactification of $Y$ \cite{Ash, Mok}. The exceptional divisors $D$ in $X$ over the singular points of $Y^{*}$ are disjoint smooth elliptic curves with negative normal bundle in $X$. Moreover, it is well known that the pair $(X, D)$ is of log-general type, again see \S \ref{ssec:Ball}. 

Recall that the Kodaira--Enriques classification gives a satisfactory classification of smooth projective surfaces in terms of their Kodaira dimension $\kappa \in \{-\infty, 0, 1, 2\}$. Moreover, we have a quite complete description of surfaces that are not of general type, i.e., the class of surfaces with $\kappa \le 1$. While most smooth toroidal compactifications of neat ball quotients are of general type with ample canonical class (Theorem A in \cite{Luca1}), it is an interesting and open question to decide which projective surfaces can arise as compactifications of ball quotients. More precisely, it would be interesting to classify all smooth projective surfaces that can be realized as smooth toroidal compactifications that are \emph{not} of general type. The first explicit examples were constructed by Hirzebruch in \cite{Hir84}, and the examples in his paper not of general type are all birational to a particular Abelian surface and hence have Kodaira dimension zero. Even though it has not been explicitly observed in the literature, it is known to experts that Hirzebruch's work implies that the volume spectrum of complex hyperbolic $2$-manifolds is unobstructed, i.e., all possible values for the volume of a finite volume ball quotient manifold are achieved.

In \cite{Mom08}, it is claimed that an irregular smooth ball quotient compactification of nonpositive Kodaira dimension is necessarily birational to an Abelian surface. In particular, this result would imply the nonexistence of smooth ball quotient compactifications birational to \emph{bielliptic} surfaces, or for simplicity bielliptic ball quotient compactifications. Unfortunately, the proof of the main result in \cite{Mom08} contains an error, and in fact the result is not true. The purpose of this paper is to produce two infinite families of explicit bielliptic ball quotient compactifications.

\begin{thm}\label{thm:MainThm1}
For any natural number $n$, there exists a smooth projective surface $X_{n}$ birational to a bielliptic surface and neat lattice $\Gam_n$ in $\PU(2, 1)$ of covolume $\frac{8}{3}\pi^{2}n$ such that $X_{n}$ is the smooth toroidal compactification of $\B^{2}/\Gam_{n}$. In particular, the family $\{\B^{2}/\Gam_{n}\}$ saturates the entire admissible volume spectrum of ball quotient surfaces with holomorphic sectional curvature $-1$.
\end{thm}

Furthermore, the associated smooth compactifications $X_{n}$ have the property that their Albanese variety is always an elliptic curve. We will prove that this in fact gives a holomorphic fibration of the ball quotient with no multiple fibers. Moreover, we will show that the free rank of $\text{H}_{1}(\B^{2}/\Gam_{n};\Z)$ is always two. Using these facts along with a topological argument due to Nori \cite{Nor83}, we obtain the following group theoretical result. For more on this circle of ideas, we refer to the results previously obtained by the second author \cite{Sto15}.

\begin{thm}\label{thm:MainThm2}
There exists an infinite sequence of nested neat lattices $\Gam_n$ in $\PU(2, 1)$ with infinite abelianization such that the commutator subgroup $[\Gam_{n}, \Gam_{n}]$ is finitely generated.
\end{thm}

The lattices $\Gam_{n}$ appear to be the first known examples of nonuniform lattices in $\PU(2, 1)$ with infinite abelianization and finitely generated commutator subgroup, and the first infinite tower of lattices of any kind. It is well-known amongst experts that the so-called Cartwright--Steger lattice \cite{CS} provides a uniform lattice with infinite abelianization and finitely generated commutator subgroup. One can prove this directly using the fact that it fibers over an elliptic curve with no multiple fibers (e.g., see \cite[Main Thm.]{CKY}). Unlike that argument, our proof does not rely on computer calculations.

We prove finite generation by showing that $[\Gam_{n}, \Gam_{n}]$ is of finite index in the kernel of a homomorphism of $\Gam_n$ onto $\Z^2$, and we prove that this kernel is finitely generated. Recall that in \cite{Sto15}, the second author constructed the first examples of lattices in $\PU(2, 1)$ that maps onto $\Z$ with finitely generated kernel. In contrast, cocompact lattices in $\PU(2, 1)$ that map onto $\Z$ with infinitely generated kernel must arise from ball quotients that are virtually fibered over a hyperbolic Riemann surface \cite{Napier}. The examples presented here are quite different, not only from a group theoretical point of view, as the associated ball quotients fiber over an elliptic curve with as generic fiber either an elliptic curve with three punctures or an elliptic curve with four punctures.

An important tool in the proof is the following general result, first stated in recent work of Kasparian and Sankaran \cite[Cor.\ 4.5]{Kasparian--Sankaran}, which allows us to conclude that the open ball quotient and its smooth toroidal compactification have the same first betti number.

\begin{thm}\label{thm:MainSTCb1}
Let $M = \B^2 / \Gam$ be a noncompact ball quotient admitting a smooth toroidal compactification $X$. Then $b_1(M) = b_1(X)$. Equivalently, the first betti number of $M$ equals its first $L^2$ betti number.
\end{thm}

Recall that the first betti number $b_1(M)$ is the rank of $H_1(M; \Q)$. For the equality between $b_1(X)$ and the first $L^2$ betti number of $M$, see \cite{Murty--Ramakrishnan} (which also proves Theorem \ref{thm:MainSTCb1} for $n$-dimensional ball quotients, $n \ge 3$). We note that, while \cite{Murty--Ramakrishnan} only studies particular arithmetic lattices, the argument for this equality relies only on Hartogs' extension theorem, and is completely general. Kasparian and Sankaran proved Theorem \ref{thm:MainSTCb1} using the fundamental group, and we give an alternate very elementary proof using the Mayer--Vietoris sequence (cf.\ \cite{Zucker}).

The paper is organized as follows. Section \ref{sec:Prelim} collects preliminary facts about the ball and its Bergman metric, smooth toroidal compactifications, the Bogomolov--Miyaoka--Yau inequality, and the basic theory of bielliptic surfaces. In Sections \S \ref{ssec:Hirz} and \S \ref{sec:Proof} we give the arguments of the proofs of Theorems \ref{thm:MainThm1} and \ref{thm:MainThm2}. Finally, in the last section, \S \ref{App}, we construct a different family of ball quotients $\B^{2}/\Lam_{n}$ which can be alternatively used in the proofs of Theorems \ref{thm:MainThm1} and \ref{thm:MainThm2}. These surfaces have smooth toroidal compactification again biholomorphic to $X_{n}$, but they are quite different from the ball quotients $\B^{2}/\Gam_{n}$. In particular, all surfaces in the family $\B^{2}/\Lam_{n}$ have exactly two cusps, while for any $n\geq1$ the surface $\B^{2}/\Gam_{n}$ always has $n+1$ cusps. It follows that we can saturate the whole volume spectrum with $2$-cusped ball quotient surfaces (it follows from \cite{Sto15} that one can do this with $4$-cusped ball quotients). This is the last result presented in this paper.

\begin{thm}\label{thm:MainThm3}
For any natural number $n$, there exists a neat lattice $\Lam_n$ in $\PU(2, 1)$ of covolume $\frac{8}{3}\pi^{2}n$ such that the associated finite volume ball quotient $\B^{2}/\Lam_{n}$ has exactly two cusps.\\
\end{thm}

\noindent\textbf{Acknowledgements}. Both authors thank the Max Planck Institute for Mathematics for its hospitality while this project was being completed. The first author thanks the MPIM for an ideal working environment throughout the whole duration of this project. The authors also thank Stefano Vidussi for comments on an early version of the paper.

\section{Preliminaries}\label{sec:Prelim}

\subsection{Smooth toroidal compactifications, their volumes and the Bogomolov--Miyaoka--Yau inequality}\label{ssec:Ball}

Let $\B^2$ be the unit ball in $\C^2$ with its Bergman metric. The group of biholomorphic isometries of $\B^2$ is isomorphic to $\PU(2, 1)$. See the \cite{Goldman} for more on its geometry. Let $\Gam \subset \PU(2, 1)$ be a nonuniform torsion-free lattice, so $\B^{2}/\Gam$ is a noncompact finite volume complex hyperbolic manifold. Suppose further that $\Gam$ is \emph{neat}, i.e., that the subgroup of $\C$ generated by the eigenvalues (of an appropriate lift to $\mathrm{SU}(2,1)$) of any fixed nontrivial element of $\Gam$ is torsion-free. In particular, any neat lattice is automatically torsion-free. This implies that $Y$ admits a particularly nice smooth toroidal compactification $X$ by adding a collection of disjoint elliptic curves $D$. Given the pair $(X, D)$, the line bundle $K_{X}+D$ is big and nef, hence $(X, D)$ is of log-general type, i.e., the Kodaira dimension of $K_{X}+D$ is maximal and the intersection number $(K_{X}+D)^{2}$ is strictly positive. See \cite{Ash}, \cite{Borel2}, and \cite{Mok} for the explicit construction and more details.

Let $Y$ be the quotient of $\B^2$ by a neat lattice and $X$ its smooth toroidal compactification. Then $X \ssm Y$ consists of a finite union of disjoint elliptic curves $T_i$, each having negative normal bundle in $X$, i.e., $T^{2}_{i}<0$ for all $i$. Let $D = \sum T_i$. Hirzebruch--Mumford proportionality \cite{Mumford} implies that
\begin{equation}\label{eq:HMp}
\overline{c}_1^2(X, D) = 3 \overline{c}_2(X, D),
\end{equation}
where $\overline{c}^{2}_{1}$ and $\overline{c}_{2}$ are the log-Chern numbers of the pair $(X, D)$. Recall that $\overline{c}_{2}(X, D)$ is the topological Euler number of $X\ssm D$ and $\overline{c}^{2}_{1}(X, D)$ is the self-intersection of the log-canonical divisor $K_{X}+D$.

Next, let $X$ be a smooth projective surface and let $D$ be a normal crossings divisor on $X$ such that $K_{X}+D$ is big and nef. Then, the logarithmic version of Yau's solution to the Calabi conjecture \cite{TianY} implies that the log-Chern numbers of the pair $(X, D)$ satisfy the so-called logarithmic Bogomolov--Miyaoka--Yau inequality
\[
\overline{c}_1^2(X, D) \le 3 \overline{c}_2(X, D).
\]
Furthermore, in the case of equality
\[
Y = X \ssm D = \B^2 / \Gam
\]
for some torsion-free lattice $\Gam \subset \PU(2, 1)$.

Summarizing, a pair $(X, D)$ with $X$ smooth and $D$ a simple normal crossings divisor with $K_{X}+D$ big and nef, hence of log-general type, achieves equality in the logarithmic Bogomolov--Miyaoka--Yau inequality if and only if it is a smooth toroidal compactification of a ball quotient surface.

\bigskip

We conclude this section by recalling Harder's generalization of the Gauss--Bonnet formula for noncompact finite volume arithmetically defined locally symmetric varieties \cite{Harder}. This formula is important for us because it gives the basic structure of the volume spectrum of ball quotients. If the holomorphic sectional curvature of a nonuniform neat ball quotient surface is normalized to be $-1$, the generalized Gauss--Bonnet formula implies that the Euler characteristic of $\B^{2}/\Gam$ is proportional to its Riemannian volume. More precisely,
\[
\chi(\B^{2}/\Gam)=\frac{3}{8\pi^{2}}\text{Vol}_{-1}(\B^{2}/\Gam),
\] 
where $\chi(\B^{2}/\Gam)$ is the topological Euler characteristic of the ball quotient $\B^{2}/\Gam$. We therefore conclude that the normalized volume spectrum can at most be the set of all positive integral multiples of $\frac{8}{3}\pi^{2}$. The volume spectrum then coincides with the set of all positive integral multiples of $\frac{8}{3}\pi^{2}$ if and only if we can find ball quotient surfaces with topological Euler number $n$ for any $n\geq 1$. If this is the case it is natural to say that the volume spectrum is unobstructed, as there are no constraints other than the obvious restriction coming from the fact that the Euler number is a positive integer.

\subsection{Bielliptic Surfaces and their basic properties}

In this section, we give the definition of \emph{bielliptic} surface, recall their place in the Kodaira--Enriques classification, and give some of their topological properties. By definition, a bielliptic surface is a minimal surface of Kodaira dimension zero and irregularity one. As shown by Bagnera and de Franchis more than a century ago, all such surfaces can be obtained as quotients of products of elliptic curves. More precisely, we have the following result for which we refer to \cite[Ch. VI]{Bea}. 

\begin{thm}[Bagnera--de Franchis, 1907]\label{bagnera}
Let $E_{\lam}$ and $E_{\tau}$ denote elliptic curves associated with the respective lattices $\Z[1, \lam]$ and $\Z[1, \tau]$ in $\C$ and $K$ be a group of translations of $E_{\tau}$ acting on $E_{\lam}$ such that $E_{\lam}/K=\pp^{1}$. Then every bielliptic surface is of the form $(E_{\lam}\times E_{\tau})/K$ where $K$ has one of the following types:
\begin{enumerate}

\item $K=\Z / 2 \Z$ acting on $E_{\lam}$ by $x\rightarrow -x$;

\item $K=\Z / 2 \Z\times\Z / 2 \Z$ acting on $E_{\lam}$ by
\[
x\rightarrow -x \quad\textrm{and}\quad x\rightarrow x+\al_{2},
\]
where $\al_{2}$ is a $2$-torsion point;

\item $K=\Z / 4 \Z$ acting on $E_\lam$ by $x\rightarrow \lam x$ with $\lam = i$;

\item $K=\Z / 4 \Z\times \Z / 2 \Z$ acting on $E_\lam$ by
\[
x\rightarrow \lam x \quad\textrm{and}\quad x\rightarrow x+\frac{1+\lam}{2},
\]
with $\lam = i$;

\item $K=\Z / 3 \Z$ acting on $E_{\lam}$ by $x\rightarrow \lam x$ with $\lam=e^{\frac{2\pi i}{3}}$;

\item $K=\Z / 3 \Z\times \Z / 3 \Z$ acting on $E_{\lam}$ by
\[
x\rightarrow \lam x \quad\textrm{and}\quad x\rightarrow x+\frac{1-\lam}{3},
\]
with $\lambda=e^{\frac{2\pi i}{3}}$;

\item $K=\Z / 6 \Z$ acting on $E_{\lam}$ by $x\rightarrow \zeta x$ with $\lam=e^{\frac{2\pi i}{3}}$ and $\zeta=e^{\frac{\pi i}{3}}$.

\end{enumerate}
\end{thm}

Given $E_{\lam}$, $E_{\tau}$ and $K$ as in Theorem \ref{bagnera}, let $B_{\lam, \tau}=(E_{\lam}\times E_{\tau})/K$ denote the associated bielliptic surface. The natural projections from $E_{\lam}\times E_{\tau}$ onto its factors give maps
\[
\psi_{1}: B_{\lam, \tau}\rightarrow E_{\lam}/ K\cong \pp^{1}, \quad \psi_{2}: B_{\lam, \tau}\rightarrow E_{\tau}/ K.
\]
Clearly, the map $\psi_{1}$ must have multiple fibers, while $E_{\tau}/ K$ is an elliptic curve and the map $\psi_{2}$ is precisely the Albanese map for the surface $B_{\lam, \tau}$. In particular, the $\Z$-rank of the homology group $\text{H}_{1}(B_{\lam, \tau}; \Z)$ is always two. Depending on the group $K$, the associated bielliptic surface $B_{\lam, \tau}$ may or may not have torsion in $\text{H}_{1}(B_{\lam, \tau}; \Z)$. For the computation of these torsion groups we refer to \cite{Ser90}. Finally, we note that, since the fundamental group is a birational invariant, we have that the first $\Z$-homology group of a blown-up bielliptic surface is the same as the first $\Z$-homology group of its minimal model.

\section{Proof of Theorem \ref{thm:MainThm1}}\label{ssec:Hirz}

In this section we explicitly construct the surfaces $X_{n}$. Let $\rho = e^{2\pi i / 3}$ and $n$ be any positive natural number. Then $\Del_n = \Z[n, 1-\rho]$ is a lattice in $\C$ for each $n \ge 1$, with $\Del = \Del_1 = \Z[\rho]$. We then have elliptic curves $G_{n}=\C / \Del_n$, and let $G = G_1$. Define $a=\frac{1-\rho}{3}$, set $A_{n} = G \times G_{n}$, let $[w, z]$ be coordinates on $A$, and consider the curves
\[
E_{1}=[z, z], \quad E_{2}=[\rho z-a, z], \quad E_{3}=[\rho^{2}z-2a, z].
\]
Then, for any $i\neq j$ we have
\[
E_{i}\cap E_{j}=\underset{0 \le m \le n-1}{\bigcup_{0 \le l \le 2}} \Big[ \frac{2}{3}+la, \frac{2}{3}+la+m \Big].
\]

Next, consider the degree three automorphism $\varphi: A_{n}\rightarrow A_{n}$ given by
\[
\varphi([w, z])=[\rho w, z+a],
\]
and let $\pi_{n}: A_{n}\rightarrow B_{n}$ be the associated degree three \'etale cover. Then $B_{n}$ is a bielliptic surface with Albanese map $\text{Alb}_{n}: B_{n} \rightarrow \C / \Z[n, a]$. Next, we observe that
\[
\varphi(E_{1})=E_{2}, \quad \varphi(E_{2})=E_{3}, \quad \varphi(E_{3})=E_{1},
\]
so that the image in $B_{n}$ of the curves $E_{1}$, $E_{2}$ and $E_{3}$ is a singular irreducible curve $C_{n}$ with exactly $n$ regular singular points of degree three. For $1\leq j\leq n$, let $F_{j}$ denote the fiber of $\text{Alb}_{n}$ over the point $[\frac{2}{3}+j-1]\in \C/\Z[n, a]$. The fibers $F_{j}$ intersect the curve $C_{n}$ in its $n$ singular points.

We now claim that by blowing up these $n$ points we obtain a smooth toroidal compactification. To see this, let $X_{n}$ be the blowup of $B_{n}$ at the $n$ singular points of the curve $C_{n}$. Then $K_{X_n}^2 = -n$ and $\chi(X_n) = n$. Let $T_{0}$ be the proper transform of $C_{n}$ in $X_{n}$ and for $i=1, ..., n$ let $T_{i}$ be the proper transform of $F_{i}$ in $X_{n}$. We have $T_0^2 = -3 n$ and $T_i^2 = -1$ for $1 \le i \le n$. Consider $D_{n}=\sum^{n}_{i=0}T_{i}$. Then, the pair $(X_{n}, D_{n})$ satisfies
\begin{align*}
\overline{c}^{2}_{1}(X_{n}, D_{n}) &= (K_{X_{n}}+D_{n})^{2} \\
&= K^{2}_{X_{n}}-T^{2}_{0} - \cdots - T^{2}_{n} \\
&= -n+3n+n \\
&= 3\overline{c}_{2}(X_{n}, D_{n}).
\end{align*}
The construction of $D_n$ implies immediately that $K_{X_n} + D_n$ is big and nef, hence $(X_{n}, D_{n})$ is of log-general type. We already saw that $(K_{X_n} + D_n)^2 > 0$. To see that $K_{X_n} + D_n$ is nef, one uses the above calculations to show that every curve not contained in the support of $D_n$ intersects $K_{X_{n}}+D_{n}$ positively. Therefore every irreducible curve on $X_n$ intersects $K_{X_n} + D_n$ nonnegatively. Since any nef divisor with positive self-intersection is big, we conclude that $K_{X_{n}}+D_{n}$ is a big and nef divisor. It follows immediately that $X_{n}\ssm D_{n}$ is biholomorphic to $\B^{2}/\Gam_{n}$ for some nonuniform torsion-free lattice $\Gam_{n}\in\PU(2, 1)$.

Next, we must show that the lattices $\Gam_{n}$ are indeed \emph{neat}. To see this, first notice that the surfaces $X_{n}$ form a tower of coverings
\[
\cdots \rightarrow X_{n}\rightarrow X_{n-1}\rightarrow \cdots \rightarrow X_{1},
\]
which induces a tower of coverings of the associated ball quotients
\[
\cdots \rightarrow \B^{2}/\Gam_{n}\rightarrow \B^{2}/\Gam_{n-1}\rightarrow \cdots \rightarrow \B^{2}/\Gam_{1}.
\]
In particular, the lattices $\Gam_{n}$ are nested in a sequence of subgroups of $\Gam_{1}$
\[
\cdots \subset \Gam_{n} \subset \Gam_{n-1} \subset \cdots \subset \Gam_{1}.
\]
As shown in \cite{DS15b}, the lattice $\Gam_{1}$ is neat, and this suffices to imply neatness of all the $\Gam_{n}$. More precisely, the surface $\B^{2}/\Gam_{1}$ corresponds to Example $2$ in Section $6$ of \cite{DS15b}.

It remains to compute the volumes of the surfaces $\B^{2}/\Gam_{n}$. Since $\chi(\B^{2}/\Gam_{n})=\chi(X_{n})=n$, we conclude that the ball quotient surfaces $\B^{2}/\Gam_{n}$ saturate the whole volume spectrum since, as recalled in \S \ref{sec:Prelim}, we have
\[
\text{Vol}_{-1}(\B^{2}/\Gam_{n})=\frac{8}{3}\pi^{2}\chi(\B^{2}/\Gam_{n}) = \frac{8}{3} \pi^2 n.
\]
The proof of Theorem \ref{thm:MainThm1} is therefore complete.

\section{Proof of Theorems \ref{thm:MainThm2} and \ref{thm:MainSTCb1}}\label{sec:Proof}

It has been noticed many times in the literature that if $X$ is a smooth toroidal compactification of the ball quotient manifold $M = \B^n / \Gam$, then $b_1(M) \ge b_1(X)$, where $b_1$ denotes the first betti number, i.e., the rank of $H_1$ with $\Q$ coefficients. For a general result, applicable not only to ball quotient compactifications, we refer to \cite[Prop.\ 2.10]{Kollar}. Further, Murty and Ramakrishnan showed that $b_1(X)$ equals the first $L^2$ betti number of $M$, and that $b_1(M) = b_1(X)$ when $n \ge 3$ \cite{Murty--Ramakrishnan}. Kasparian and Sankaran \cite[Cor.\ 4.5]{Kasparian--Sankaran} proved that $b_1(M) = b_1(X)$ when $n = 2$ using the fundamental group, and we now give a very elementary proof of that result (cf.\ \cite{Zucker}).

\begin{pf}[Proof of Theorem \ref{thm:MainSTCb1}]
Consider a smooth toroidal compactification $(X, D)$ of the ball quotient manifold $M$, where $D$ consists of $k$ disjoint elliptic curves. Choose an open neighborhood $U$ of $D$ consisting of $k$ mutually disjoint open sets, one for each irreducible component of $D$. In what follows, we use $H_i$ to denote homology with $\Q$ coefficients, since we do not care about torsion.

Then $U$ deformation retracts on $k$ disjoint $2$-tori. Thus, we have:
\[
H_i(U) =
\begin{cases}
\Q^k &  i=0, 2\\
\Q^{2 k} & i = 1\\
\{0\} & i=3, 4
\end{cases}
\]
Now define $V = U \ssm D$, so $V$ deformation retracts on the disjoint union of $k$ closed Nil $3$-manifolds. We then have:
\begin{equation*}
H_{i}(V) =
\begin{cases}
\Q^k & i=0, 3\\
\Q^{2 k} & i = 1,2 \\
\{0\} & i=4
\end{cases}
\end{equation*}
Recall that any closed Nil 3-manifold $N$ arising as the cusp cross section in a smooth toroidal compactification is a circle bundle over a $2$-torus satisfying
\[
H_{1}(N^{3}; \Z)=\Z^{2}\oplus\text{Torsion},
\]
with the torsion part depending on the specific nilmanifolds, while $H_{2}(N; \Z)$ is always torsion-free equal to $\Z^{2}$ by duality. For these facts we refer to \cite{Thurston}.

Next, we apply the Mayer--Vietoris sequence to $X = M \cup U$. First, we consider
\[
\cdots\rightarrow H_1(X) \rightarrow H_0(V) \rightarrow H_0(M)\oplus H_0(U) \rightarrow H_0(X)\rightarrow \{0\}
\]
which gives
\[
\cdots\rightarrow H_1(X) \rightarrow \Q^k \rightarrow \Q \oplus \Q^k \rightarrow \Q \rightarrow \{0\}
\]
and it follows that $H_1(X) \to H_0(V)$ is zero. This gives an exact sequence
\[
\cdots\rightarrow H_2(X) \rightarrow H_{1}(V) \rightarrow H_{1}(M)\oplus H_{1}(U) \rightarrow H_{1}(X)\rightarrow \{0\}
\]
that becomes
\[
\cdots\rightarrow H_2(X) \rightarrow \Q^{2 k} \rightarrow H_{1}(M)\oplus \Q^{2 k} \rightarrow H_{1}(X)\rightarrow \{0\}
\]
It follows immediately that $b_1(M) + 2 k = 2 k - \ell + b_1(X)$, where $\ell$ is the dimension of the image of the map $H_2(X) \to H_1(V)$. In other words, $b_1(M) = b_1(X) - \ell$. However, as discussed above, it is well-known that $b_1(M) \ge b_1(X)$, so $\ell = 0$ and the theorem follows.
\end{pf}

%

\begin{rem}
It is not necessarily the case that $H_1(M; \Z) \cong H_1(X; \Z)$. Indeed, we can construct examples, closely related to the examples in this paper, where $H_1(M; \Z)$ has torsion but $H_1(X; \Z)$ is torsion-free.
\end{rem}

\begin{rem}
The remainder of the Mayer--Vietoris sequence reduces to an exact sequence:
\[
\{0\} \to \Q^{k - 1} \to H_3(M) \to H_3(X) \to \Q^{2 k} \to H_2(M) \oplus \Q^k \to H_2(X) \to \{0\}
\]
The image in $H_3(M)$ of $\Q^{k - 1}$ is generated by any $k - 1$ of the cusp cross-sections of $M$ (the $k^{th}$ is clearly linearly dependent, since the union of all the cusp cross-sections obviously bounds). For example, one can then conclude that the betti numbers of $M$ satisfy:
\begin{align*}
b_3(M) & \ge k - 1 \\
b_2(M) - b_3(M) &= 1 - b_3(X) + b_2(X)
\end{align*}
\end{rem}

We are now ready to prove Theorem \ref{thm:MainThm2}.


\begin{proof}[Proof of Theorem \ref{thm:MainThm2}]
Recall the surfaces $X_n$ from Theorem \ref{thm:MainThm1}. For any $n$, let $Alb_{n}: B_{n}\rightarrow E_{n}=\C/\Z[n, a]$ be the Albanese map and let $\pi_{n}: X_{n}\rightarrow B_{n}$ be the blowup map. Consider $\psi_{n}: M_{n}\rightarrow E_{n}$, which is the composition $\psi_{n}=Alb_{n}\circ\pi_{n}\circ i_{n}$, where $i_{n}: M_{n}\rightarrow X_{n}$ is the inclusion. We then obtain a surjective morphism
\[
(\psi_{n})_{*}: \pi_{1}(M_{n})\rightarrow \pi_{1}(E_{n})\cong\Z^{2}.
\]

The generic fiber $F_{n}$ of the surjective fibration $\psi_{n}: M_{n}\rightarrow E_{n}$ is a reduced torus with three punctures. The singular fibers of this fibrations are reduced smooth rational curves with four punctures. Note that there are exactly $n$ singular fibers corresponding with each of the $n$ exceptional divisors in $X_{n}$.

By Lemma 1.5 in \cite{Nor83}, the sequence
\[
\pi_{1}(F_{n})\rightarrow\pi(M_{n})\rightarrow\pi(E_{n})\cong\Z^{2}\rightarrow 1
\]
is exact. We therefore conclude that $\text{Ker}((\psi_{n})_{*})$ is finitely generated for any $n$. Also, the free rank of $\text{H}_{1}(M_{n}; \Z)$ is always two by Theorem \ref{thm:MainSTCb1}. On the other hand, we have that $\Gam_{n}/\text{Ker}((\psi_{n})_{*})\cong \Z^{2}$ for any $n$. It follows that the commutator subgroup $[\Gam_{n}, \Gam_{n}]$ is of finite index in $\text{Ker}((\psi_{n})_{*})$. Since finite index subgroups of finitely generated groups are finitely generated, the proof is complete.
\end{proof}

\begin{rem}
It follows from arguments in \cite{Kap} that $[\Gam_n, \Gam_n]$, while finitely generated, cannot be finitely presented.
\end{rem}

\section{Two-cusped examples}\label{App}

In this section, we explicitly construct a second distinct family of ball quotients $\B^{2}/\Lam_{n}$ which can be alternatively used in the proofs of the main theorems presented in this paper. These ball quotients have toroidal compactifications biholomorphic to the nonminimal bielliptic surfaces $X_{n}$ constructed above. In other words, for any $n$, even if the ball quotients $\B^{2}/\Gam_{n}$ and $\B^{2}/\Lam_{n}$ are not biholomorphic, they nevertheless have biholomorphic smooth toroidal compactifications. For many more examples and a detailed study of the multiple realizations problem of varieties as ball quotient compactifications, we refer to \cite{DS15a}.

The most natural way to distinguish the surfaces $\B^{2}/\Lam_{n}$ from the surfaces $\B^{2}/\Gam_{n}$ is to look at their cusps. In particular, we will show that all of surfaces in $\B^{2}/\Lam_{n}$ have exactly two cusps, while we already know that for any $n\geq 1$ the surface $\B^{2}/\Gam_{n}$ has $n+1$ cusps. This is clearly enough to show that the two families are distinct for $n\geq 2$. For $n=1$ this is still the case but the argument is different and we refer to end of this section for details.

\bigskip

As before, let $\rho = e^{2\pi i / 3}$ and let $n$ be a any positive natural number. Let the lattices $\Del_n$ in $\C$ be as above with $G_{n}=\C / \Del_n$ the associated elliptic curves. Again set $a=\frac{1-\rho}{3}$,  $A_{n} = G \times G_{n}$, let $[w, z]$ be coordinates on $A$, and consider the curves:
\begin{align*}
E_{1}=[z, z], \quad E_{2}=[\rho z-a, z], \quad E_{3}=[\rho^{2}z-2a, z].
\end{align*}
Recall that if $i \neq j$, then
\[
E_{i}\cap E_{j} = \underset{0 \le l \le 2}{\bigcup_{0 \le m \le n - 1}}\Big[\frac{2}{3}+la, \frac{2}{3}+la+m\Big].
\]

We again consider $\varphi: A_{n}\rightarrow A_{n}$ given by
\[
\varphi([w, z])=[\rho w, z+a],
\]
which satisfies
\[
\varphi(E_{1})=E_{2}, \quad \varphi(E_{2})=E_{3}, \quad \varphi(E_{3})=E_{1},
\]
and let $\pi_{n}: A_{n}\rightarrow B_{n}$ the associated degree three \'etale cover. Then $B_{n}$ is a bielliptic surface with Albanese map $\text{Alb}_{n}: B_{n}\rightarrow \C/\Z[n, a]$. The image in $B_{n}$ of the curves $E_{1}$, $E_{2}$ and $E_{3}$ is a singular irreducible curve $C_{n}$ with exactly $n$ regular singular points of degree three.

Now we diverge from the previous construction. Consider
\begin{align*}
H_{1}= \Big[\frac{2}{3}, z\Big], \quad H_{2}=\Big[\frac{2}{3}+a, z\Big], \quad H_{3}=\Big[\frac{2}{3}+2a, z\Big],
\end{align*}
and observe that
\[
\varphi(H_{1})=H_{2}, \quad \varphi(H_{2})=H_{3}, \quad \varphi(H_{3})=H_{1}.
\]
To prove this, it suffices to compute that $\frac{2}{3}+a=\frac{2\rho}{3}$ and $\frac{2}{3}+2a=\frac{2\rho^{2}}{3}$ modulo $\Del$. In particular, the image in $B_{n}$ of the curves $H_{1}$, $H_{2}$ and $H_{3}$ is a single smooth elliptic curve $F_{n}$ that is a smooth fiber of the map $\phi_{n}: B_{n}\rightarrow G/\Z_{3}\cong \pp^{1}$.

The curves $C_{n}$ and $F_{n}$ meet transversally exactly in the $n$ singular points of $C_{n}$. Thus, let $X_{n}$ be the blowup of $B_{n}$ at the $n$ singular points of $C_{n}$, $T_{0}$ denote the proper transform of $C_{n}$ in $X_{n}$, and $T_{1}$ be the proper transform of $F_{n}$ in $X_{n}$. Set $D_{n}=\sum^{1}_{i=0}T_{i}$. Again, one checks that $\overline{c}^{2}_{1}(X_{n}, D_{n}) = 3\overline{c}_{2}(X_{n}, D_{n})$ with $K_{X_n} + D_n$ big and nef, so the pair $(X_{n}, D_{n})$ is a smooth toroidal compactification. Thus, $X_{n}\ssm D_{n}$ is biholomorphic to $\B^{2}/\Lam_{n}$ for some nonuniform torsion-free lattice in $\Lam_{n}\in\PU(2, 1)$. It follows immediately from the construction that $\B^{2}/\Lam_{n}$ has exactly two cusps.

The surface $\B^{2}/\Lam_{1}$ corresponds to Example $3$ in Section $6$ of \cite{DS15b}, and it follows as with the previous examples that the lattices $\Lam_n$ are neat. The volume calculations are also exactly the same. The family $\B^{2}/\Lam_{n}$ appears to be the first family of $2$-cusped ball quotients that saturate the entire volume spectrum (see \cite{Sto15} for $4$-cusped examples). This completes the proof of Theorem \ref{thm:MainThm3}. \qed


\end{document}